\documentclass[a4paper, 10pt]{article}

\usepackage[UKenglish]{babel}
\usepackage[utf8]{inputenc}

\usepackage[T1]{fontenc}
\usepackage{lmodern}

\usepackage{amsmath, amssymb, amsfonts, mathtools, amsthm}
\usepackage{url}
\usepackage{authblk}
\usepackage[ruled, vlined, linesnumbered]{algorithm2e}
\usepackage{booktabs}

\usepackage{graphicx}

\newtheorem{proposition}{Proposition}

\providecommand{\abs}[1]{\lvert#1\rvert}
\providecommand{\norm}[1]{\lVert#1\rVert}

\newcommand{\R}{\mathbb{R}}
\newcommand{\N}{\mathbb{N}}

\DeclareMathOperator*{\argmin}{arg\,min}

\newcommand{\BigO}[1]{\ensuremath{O(#1)}}

\newcommand{\Id}{\ensuremath{\mathrm{Id}}}

\begin{document}

\title{A Dynamic Programming Solution to Bounded Dejittering Problems}

\author{Lukas F. Lang}
\affil{\footnotesize Johann Radon Institute for Computational and Applied Mathematics, Austrian Academy of Sciences, Altenberger Stra{\ss}e~69, 4040 Linz, Austria}

\date{}
\maketitle

\begin{abstract}
\noindent
We propose a dynamic programming solution to image dejittering problems with bounded displacements and obtain efficient algorithms for the removal of line jitter, line pixel jitter, and pixel jitter.
\end{abstract}

\section{Introduction}

In this article we devise a dynamic programming (DP) solution to dejittering problems with bounded displacements.
In particular, we consider instances of the following image acquisition model.
A $D$-dimensional image $u^{\delta}: \Omega \to \R^{D}$ defined on a two-dimensional domain $\Omega \subset \R^{2}$ is created by the equation
\begin{equation}
	u^{\delta} = u \circ \Phi,
\label{eq:model}
\end{equation}
where $u: \Omega \to \R^{D}$ is the original, undisturbed image, $\Phi: \Omega \to \Omega$ is a displacement perturbation, and $\circ$ denotes function composition.
Typically, $\Phi = (\Phi_{1}, \Phi_{2})^{\top}$ is a degradation generated by the acquisition process and is considered random.
In addition, $u^{\delta}$ may exhibit additive noise $\eta$.
The central theme of this article is displacement error correction, which is to recover the original image solely from the corrupted image.

One particularly interesting class are dejittering problems.
Jitter is a common artefact in digital images or image sequences and is typically attributed to inaccurate timing during signal sampling, synchronisation issues, or corrupted data transmission~\cite{Kok98,Lab03}.
Most commonly, it is observed as \emph{line jitter} where entire lines are mistakenly shifted left or right by a random displacement.
As a result, shapes appear jagged and unappealing to the viewer.
Other---equally disturbing---defects are \emph{line pixel jitter} and \emph{pixel jitter}.
The former type is due to a random shift of each position in horizontal direction only, while the latter is caused by a random displacement in $\R^{2}$.
See Fig.~\ref{fig:jitter} for examples.

\begin{figure}[t]
	\centering
	\includegraphics[width=0.24\textwidth]{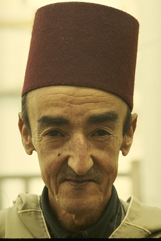}
	\hfill
	\includegraphics[width=0.24\textwidth]{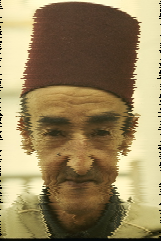}
	\hfill
	\includegraphics[width=0.24\textwidth]{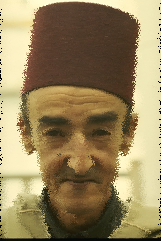}
	\hfill
	\includegraphics[width=0.24\textwidth]{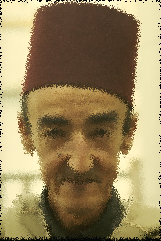}
\caption{Original image, line jitter ($\rho = 4 \, \text{pixels}$), line pixel jitter ($\rho = 5 \, \text{pixels}$), pixel jitter corruption ($\rho = 5 \, \text{pixels}$).}
\label{fig:jitter}
\end{figure}

The problem of \emph{dejittering} is to reverse the observed effect and has been studied in numerous works, see \cite{DonPatSchOek15,KanShe06,KanShe07,Kok98,Lab03,LenSch11,Nik09a,Nik09,She04}.
Recent efforts either deal with finite-dimensional minimisation problems~\cite{Lab03,Nik09a,Nik09} or rely on an infinite-dimensional setting~\cite{DonPatSchOek15,LenSch11,She04}.
Typically, variational approaches such as \cite{DonPatSchOek15,LenSch11} are based on a linearisation of \eqref{eq:model} and try to directly infer the true image.
Alternatively, as done in~\cite{She04}, one can alternate between finding the displacement and inferring the original image.
Both approaches typically enforce a certain regularity of the reconstructed image.

In this article, we investigate efficient solutions to dejittering models introduced in \cite{DonPatSchOek15,LenSch11}.
However, we assume the magnitude of each component of the displacement $x - \Phi(x)$, where $x = (x_{1}, x_{2})^{\top} \in \Omega$, to be bounded by a constant $\rho > 0$. That is,
\begin{equation}
	\norm{x_{i} - \Phi_{i}(x)}_{L^{\infty}(\Omega)} \le \rho.
\label{eq:bd}
\end{equation}
The main idea is to assume that \eqref{eq:model} can be inverted (locally) by reconstructing the original value from a small neighbourhood.
Even though not guaranteed theoretically, this approach is found to work surprisingly well for Gaussian distributed displacements of zero mean.
A possible explanation is that the original value at a certain position $x \in \Omega$ is likely to occur in the close vicinity of $x$.
Moreover, it does not require derivatives of the disturbed data, which typically occur during linearisation of \eqref{eq:model}, see \cite{DonPatSchOek15,LenSch11}.
The obvious drawback is that $u(x)$ can only take values which appear in $u^{\delta}$ within a small neighbourhood of $x$.
As a result, its capabilities are limited in the presence of noise.

We build on previous work by Laborelli~\cite{Lab03} and Nikolova~\cite{Nik09a,Nik09}, and utilise DP for the numerical solution.
For the removal of line jitter, we extend the algorithm in \cite{Lab03} to include regularisation of the displacement, yielding a stabler reconstruction.
In comparison to the greedy approach in \cite{Nik09a,Nik09} we are able to recover a global minimiser to the one-dimensional non-linear and possibly non-convex minimisation problem formulated in Sec.~\ref{sec:model}.
For the case of line pixel jitter, we rewrite the problem into a series of independent minimisation problems, each of which can be solved optimally via DP.
For pixel jitter removal we follow a different strategy as the regularisation term in the considered functional prohibits a straightforward decomposition into simpler subproblems.
We employ block coordinate descent, which is an iterative method and is guaranteed to converge energy-wise \cite{CheKol14}.
All of our algorithms generalise to $D$-dimensional images defined on $\R^{2}$.
Moreover, generalisation to regularisation functionals involving higher-order derivatives of the sought image and to higher-order discretisation accuracy is straightforward.
Table~\ref{tab:summary} summarises the results of this work.

\begin{table}[t]
\centering
\begin{tabular}{lllll}
	Algorithm & Time & Memory & Comment \\
	\midrule
	Line jitter & $\BigO{mn\rho^{q}}$ & $\BigO{n\rho^{q-1}}$ & same as \cite{Lab03} but with regularisation\\
	\midrule
	Line pixel jitter & $\BigO{mn\rho^{q}}$ & $\BigO{n\rho^{q-1}}$ & \\
	\midrule	
	Pixel jitter & $\BigO{mn\rho^{2q}}$ & $\BigO{n\rho^{2(q-1)}}$ & per-iteration complexity \\\\
\end{tabular}
\caption{Summary of the algorithms. $m$ and $n$ denote the columns and the rows of an image, respectively, $\rho \in \N$ is the maximum displacement, and $q \in \N$ is the order of the highest occurring derivative. For instance, $q = 2$ for first-order derivatives. The dimension $D$ of the image is assumed constant.}
\label{tab:summary}
\end{table}

\paragraph{\textbf{Notation.}}
Let $\Omega = [0, W] \times [0, H] \subset \mathbb{R}^{2}$ be a two-dimensional domain.
For $x = (x_{1}, x_{2})^{\top} \in \R^{2}$, the $p$-th power of the usual $p$-norm of $\R^{2}$ is denoted by $\norm{x}_{p}^{p} = \sum_{i} \abs{x_{i}}^{p}$.
For $D \in \N$, we denote by $u: \Omega \rightarrow \mathbb{R}^{D}$, respectively, by $u^{\delta}: \Omega \rightarrow \mathbb{R}^{D}$ the unknown original and the observed, possibly corrupted, $D$-dimensional image.
A vector-valued function $u = (u_{1}, \dots, u_{D})^{\top}$ is given in terms of its components.
We write $\partial_{i}^{k} u$ for the $k$-th partial derivative of $u$ with respect to $x_{i}$ and, for simplicity, we write $\partial_{i} u$ for $k = 1$.
For $D = 1$, the spatial gradient of $u$ in $\R^{2}$ is $\nabla u = (\partial_{1} u, \partial_{2} u)^{\top}$ and for $D = 3$ it is given by the matrix $\nabla u = (\partial_{i} u_{j})_{ij}$.
In the former case, its $p$-norm is simply $\norm{\nabla u}_{p}$ and in the latter case it is given by $\norm{\nabla u}_{p}^{p} = \sum_{j=1}^{D} \sum_{i=1}^{2} (\partial_{i} u_{j})^{p}$.
For a function $f: \Omega \to \R^{D}$ and $1 \le p < \infty$ we denote the $p$-power of the norm of $L^{p}(\Omega, \R^{D})$ by $\norm{f}_{L^{p}(\Omega)}^{p} = \int_{\Omega} \norm{f(x)}_{p}^{p} \; dx$.
Moreover, $\norm{f}_{L^{\infty}(\Omega)}$ denotes the essential supremum norm.
A continuous image gives rise to a discrete representation $u_{i,j} \in \R^{D}$ of each pixel.
A digital image is stored in matrix form $u \in \R^{m \times n \times D}$, where $m$ denotes the number of columns arranged left to right and $n$ the number of rows stored from top to bottom.
\section{Problem Formulation} \label{sec:model}

Let $u^{\delta}: \Omega \to \R^{D}$ be an observed and possibly corrupted image generated by \eqref{eq:model}.
We aim to reconstruct an approximation of the original image $u: \Omega \to \R^{D}$.
The main difficulty is that $\Phi^{-1}$ might not exist and that $u^{\delta}$ might exhibit noise.
Lenzen and Scherzer \cite{LenSch11} propose to find a minimising pair $(u, \Phi)$ to the energy
\begin{equation}
	\int_{\Omega} \norm{\Phi(x) - x}_{2}^{2} \; dx + \alpha \mathcal{R}(u)
\label{eq:dispreg}
\end{equation}
such that $(u, \Phi)$ satisfies \eqref{eq:model}.
Here, $\mathcal{R}(u)$ is a regularisation functional and $\alpha > 0$ is a parameter.
In what follows, we consider one exemplary class of displacements which arise in dejittering problems.
They are of the form
\begin{equation}
	\Phi = \Id + d,
\label{eq:phi}
\end{equation}
with $d: \Omega \to \Omega$ depending on the particular jitter model.
Typically, $\mathcal{R}(u)$ is chosen in accordance with $d$.
We assume that $d$ is Gaussian distributed around zero with $\sigma^{2}$ variance and whenever $x + d$ lies outside $\Omega$ we typically have $u^{\delta}(x) = 0$.
In order to approximately reconstruct $u$ we will assume that, for every $x \in \Omega$, there exists
\begin{equation*}
	d(x) = \arg\inf \{ \norm{v}_{2} \mid v \in \R^{2}, u(x) = u^{\delta}(x - v) \}.
\end{equation*}
In other words, we can invert \eqref{eq:model} and locally reconstruct $u$ by finding $d$.
While this requirement trivially holds true for line jitter under appropriate treatment of the boundaries, it is not guaranteed in the cases of line pixel jitter and pixel jitter.
Moreover, as a consequence of \eqref{eq:phi} and \eqref{eq:bd} we have, for $i \in \{1, 2\}$,
\begin{equation*}
	\norm{\Phi_{i}(x) - x_{i}}_{L^{\infty}(\Omega)} = \norm{d_{i}}_{L^{\infty}(\Omega)} \le \rho.
\end{equation*}

\paragraph{\textbf{Line Jitter.}}
In this model, the corrupted image is assumed to be created as
\begin{equation}
	u^{\delta}(x_{1}, x_{2}) = u(x_{1} + d(x_{2}), x_{2}) + \eta(x_{1}, x_{2}),
\label{eq:lj}
\end{equation}
where $d: [0, H] \rightarrow \mathbb{R}$ is a random displacement and $\eta: \Omega \to \R^{D}$ is typically Gaussian white noise.
The corruption arises from a horizontal shift of each line by a random amount, resulting in visually unappealing, jagged shapes.
Assuming zero noise, \eqref{eq:lj} can be inverted within $[\rho, W - \rho] \times [0, H]$ given $d$.
The original image is thus given by
\begin{equation}
	u(x_{1}, x_{2}) = u^{\delta}(x_{1} - d(x_{2}), x_{2}).
\label{eq:ljinverted}
\end{equation}
For $\eta \not\equiv 0$ additional image denoising is required, see e.g. \cite[Chap.~4]{SchGraGroHalLen09} for standard methods of variational image denoising.
We minimise the energy
\begin{equation}
	\mathcal{E}_{\alpha, p}^{k}(d) \coloneqq \alpha \norm{d}_{L^{2}([0, H])}^{2} + \sum_{\ell=1}^{k} \iint_{\Omega} \norm{\partial_{2}^{\ell} u^{\delta}(x_{1} - d(x_{2}), x_{2})}_{p}^{p} \; dx_{1} \, dx_{2},
\label{eq:ldj}
\end{equation}
subject to $\norm{d}_{L^{\infty}([0, H])} \le \rho$.
The first term in \eqref{eq:ldj} is suitable for displacements which are Gaussian distributed around zero.
It prevents the reconstruction from being fooled by dominant vertical edges and effectively removes a constant additive displacement, resulting in a centred image.
The second term utilises identity \eqref{eq:ljinverted} and penalises the sum of the magnitudes of vertical derivatives of the reconstructed image up to $k$-th order.
Here, $\alpha \ge 0$ is a regularisation parameter and $p > 0$ is an exponent.
The proposed framework for the solution of \eqref{eq:ldj} is more general and allows a different exponent $p_{\ell}$ and an individual weight for each term in the sum.
Moreover, any other norm of $d$ might be considered.

We restrict ourselves to discretisations of $\mathcal{E}_{\alpha, p}^{1}$ and $\mathcal{E}_{\alpha, p}^{2}$ and assume that images are piecewise constant, are defined on a regular grid, and that all displacements are integer.
Then, for every $j \in \{1, \dots, n\}$, we seek $d_{j} \in \mathcal{L}$ with $\mathcal{L} \coloneqq \{-\rho, \dots, \rho\}$.
By discretising with backwards finite differences we obtain
\begin{align}
	\mathcal{E}_{\alpha, p}^{1}(d) & \approx \sum_{j=1}^{n} \alpha \abs{d_{j}}^{2} + \sum_{j=2}^{n} \sum_{i=1}^{m} \norm{u^{\delta}_{i - d_{j}, j} - u^{\delta}_{i - d_{j-1}, j-1}}_{p}^{p}, \label{eq:ldj1} \\
	\mathcal{E}_{\alpha, p}^{2}(d) & \approx \mathcal{E}_{\alpha, p}^{1}(d) + \sum_{j=3}^{n} \sum_{i=1}^{m} \norm{u^{\delta}_{i - d_{j}, j} - 2u^{\delta}_{i - d_{j-1}, j-1} + u^{\delta}_{i - d_{j-2}, j - 2}}_{p}^{p}. \label{eq:ldj2}
\end{align}

\paragraph{\textbf{Line Pixel Jitter.}}
Images degraded by line pixel jitter are generated by
\begin{equation}
	u^{\delta}(x_{1}, x_{2}) = u(x_{1} + d(x_{1}, x_{2}), x_{2}) + \eta(x_{1}, x_{2}),
\label{eq:lpj}
\end{equation}
where $d: \Omega \rightarrow \mathbb{R}$ now depends on both $x_{1}$ and $x_{2}$.
As before, the displacement is in horizontal direction only.
Images appear pixelated and exhibit horizontally fringed edges.
In contrast to line jitter, in the noise-free setting one is in general not able to reconstruct the original image solely from $u^{\delta}$, unless $d(\cdot, x_{2}): [0, W] \to \R$ is bijective on $[0, W]$ for every $x_{2} \in [0, H]$.
As a remedy, we utilise the fact that $d$ is assumed to be independent (in $x_{1}$) and identically Gaussian distributed around zero.
The idea is that the original value $u(x_{1}, x_{2})$, or a value sufficiently close, at $(x_{1}, x_{2})$ is likely be found in a close neighbourhood with respect to the $x_{1}$-direction.
We assume that
\begin{equation}
	d(x_{1}, x_{2}) = \arg\inf\{\abs{v} \mid v \in \R, u(x_{1}, x_{2}) = u^{\delta}(x_{1} - v, x_{2}) \}
\label{eq:lpdjinverted}
\end{equation}
exists and that $\norm{d}_{L^{\infty}(\Omega)} \le \rho$.
Clearly, it is not unique without further assumptions, however, finding one $d$ is sufficient.
We utilise \eqref{eq:lpdjinverted} and minimise
\begin{equation}
	\mathcal{F}_{\alpha, p}^{k}(d) \coloneqq \alpha \norm{d}_{L^{2}(\Omega)}^{2} + \sum_{\ell=1}^{k} \iint_{\Omega} \norm{\partial_{2}^{\ell} u^{\delta}(x_{1} - d(x_{1}, x_{2}), x_{2})}_{p}^{p} \; dx_{1} \, dx_{2}
\label{eq:lpdj}
\end{equation}
subject to $\norm{d}_{L^{\infty}(\Omega)} \le \rho$.
Again, $p > 0$ and $\alpha \ge 0$.
In contrast to before, we decompose the objective into a series of minimisation problems, which then can be solved independently and in parallel by DP.
To this end, let us rewrite
\begin{align*}
	\mathcal{F}_{\alpha}^{k}(d) & = \int_{0}^{W} \int_{0}^{H} \left( \alpha \abs{d(x_{1}, x_{2})}^{2} + \sum_{\ell=1}^{k} \norm{\partial_{2}^{\ell} u^{\delta}(x_{1} - d(x_{1}, x_{2}), x_{2})}_{p}^{p} \right) \; dx_{2} \; dx_{1}.
\end{align*}

As before we consider derivatives up to second order.
Assuming piecewise constant images defined on a regular grid we seek, for each $(i, j) \in \{1, \dots, m\} \times \{1, \dots, n\}$, a displacement $d_{i, j} \in \mathcal{L}$ with $\mathcal{L} \coloneqq \{-\rho, \dots, \rho\}$.
The finite-dimensional approximations of $\mathcal{F}_{\alpha, p}^{1}$ and $\mathcal{F}_{\alpha, p}^{2}$ hence read
\begin{align}
	\mathcal{F}_{\alpha, p}^{1}(d) & \approx \sum_{i=1}^{m} \left( \sum_{j=1}^{n} \alpha \abs{d_{i, j}}^{2} + \sum_{j=2}^{n} \norm{u^{\delta}_{i - d_{i, j}, j} - u^{\delta}_{i - d_{i, j-1}, j-1}}_{p}^{p} \right), \label{eq:lpdj1} \\
	\mathcal{F}_{\alpha, p}^{2}(d) & \approx \mathcal{F}_{\alpha, p}^{1}(d) + \sum_{i=1}^{m} \sum_{j=3}^{n} \norm{u^{\delta}_{i - d_{i, j}, j} - 2u^{\delta}_{i - d_{i, j-1}, j-1} + u^{\delta}_{i - d_{i, j-2}, j - 2}}_{p}^{p}. \label{eq:lpdj2}
\end{align}

\paragraph{\textbf{Pixel Jitter.}}
An image corrupted by pixel jitter is generated by
\begin{equation}
	u^{\delta}(x_{1}, x_{2}) = u(x_{1} + d_{1}(x_{1}, x_{2}), x_{2} + d_{2}(x_{1}, x_{2})) + \eta(x_{1}, x_{2}),
\label{eq:pj}
\end{equation}
where $d = (d_{1}, d_{2})^{\top}$, $d_{i}: \Omega \to \R$, is now a vector-valued displacement.
Edges in degraded images appear pixelated and fringed in both directions.
Unless the displacement $d$ is bijective from $\Omega$ to itself and $\eta \equiv 0$, there is no hope that $u$ can be perfectly reconstructed from $u^{\delta}$.
However, we assume the existence of
\begin{equation*}
	d(x) = \arg\inf\{\norm{v}_{2} \mid v \in \R^{2}, u(x) = u^{\delta}(x - v) \}
\end{equation*}
such that $\norm{d_{i}}_{L^{\infty}(\Omega)} \le \rho$, for $i \in \{1, 2\}$.
For $p > 0$ and $\alpha \ge 0$, we minimise
\begin{equation*}
	\mathcal{G}_{\alpha, p}(d) \coloneqq \alpha \norm{d}_{L^{2}(\Omega)}^{2} + \int_{\Omega} \norm{\nabla (u^{\delta}(x - d(x))}_{p}^{p} \; dx
\end{equation*}
subject to $\norm{d_{i}}_{L^{\infty}(\Omega)} \le \rho$, $i \in \{1, 2\}$.
In contrast to before, we only consider first-order derivatives of the sought image.

Assuming piecewise constant images on a regular grid and integer displacements, we seek for each $(i, j) \in \{1, \dots, m\} \times \{1, \dots, n\}$ an offset $d_{i, j} \in \mathcal{L}$ with $\mathcal{L} \coloneqq \{-\rho, \dots, \rho\}^{2}$.
In further consequence, we obtain
\begin{equation}
\begin{aligned}
	\mathcal{G}_{\alpha, p}(d) & \approx \sum_{i=1}^{m} \sum_{j=1}^{n} \alpha \norm{d_{i, j}}_{2}^{2} + \sum_{i=2}^{m} \sum_{j=1}^{n} \norm{u^{\delta}_{(i, j) - d_{i, j}} - u^{\delta}_{(i - 1, j) - d_{i - 1, j}}}_{p}^{p} \\
	& \qquad + \sum_{i=1}^{m} \sum_{j=2}^{n} \norm{u^{\delta}_{(i, j) - d_{i, j}} - u^{\delta}_{(i, j - 1) - d_{i, j - 1}}}_{p}^{p}.
\end{aligned}
\label{eq:pdj}
\end{equation}
\section{Numerical Solution}

\paragraph{\textbf{Dynamic Programming on a Sequence.}}
Suppose we are given $n \in \N$ elements and we aim to assign to each element $i$ a label from its associated space of labels $\mathcal{L}_{i}$.
Without loss of generality, we assume that all $\mathcal{L}_{i}$ are identical and contain finitely many labels.
A labelling is denoted by $x = (x_{1}, \dots, x_{n})^{\top} \in \mathcal{L}^{n}$, where $x_{i} \in \mathcal{L}$ is the label assigned to the $i$-th element.

Let us consider the finite-dimensional minimisation problem
\begin{equation}
	\min_{x \in \mathcal{L}^{n}} \; \sum_{i=1}^{n} \varphi_{i}(x_{i}) + \sum_{i=2}^{n} \psi_{i-1, i}(x_{i-1}, x_{i}).
\label{eq:fdmin}
\end{equation}
We denote a minimiser by $x^{*}$ and its value by $E(x^{*})$.
Here, $\varphi_{i}(x_{i})$ is the penalty of assigning the label $x_{i} \in \mathcal{L}$ to element $i$, whereas $\psi_{i-1, i}(x_{i-1}, x_{i})$ is the cost of assigning $x_{i-1}$ to the element $i-1$ and $x_{i}$ to $i$, respectively.
Several minimisers of \eqref{eq:fdmin} might exist but finding one is sufficient for our purpose.
Energies \eqref{eq:fdmin} typically arise from the discretisation of computer vision problems such as one-dimensional signal denoising, stereo matching, or curve detection.
We refer to \cite{FelZab11} for a comprehensive survey and to \cite{CorLeiRivSte09} for a general introduction to DP.

The basic idea for solving \eqref{eq:fdmin} is to restate the problem in terms of smaller subproblems.
Let $\abs{\mathcal{L}}$ denote the cardinality of $\mathcal{L}$.
Then, for $j \le n$ and $\ell \le \abs{\mathcal{L}}$, we define $\mathrm{OPT}(j, \ell)$ as the minimum value of the above minimisation problem~\eqref{eq:fdmin} over the first $j$ elements with the value of the last variable $x_{j}$ being set to the $\ell$-th label.
That is,
\begin{equation*}
	\mathrm{OPT}(j, \ell) \coloneqq \min_{x \in \mathcal{L}^{j-1}} \; \sum_{i=1}^{j} \varphi_{i}(x_{i}) + \sum_{i=2}^{j} \psi_{i-1, i}(x_{i-1}, x_{i}).
\end{equation*}
Moreover, we define $\mathrm{OPT}(0, \ell) \coloneqq 0$ for all $\ell \le \abs{\mathcal{L}}$ and $\psi_{0, 1} \coloneqq 0$, and show the following recurrence:

\begin{proposition}
Let $x_{\ell} \in \mathcal{L}$ denote the label of the $j$-th element.
Then,
\begin{equation*}
	\mathrm{OPT}(j, \ell) = \varphi_{j}(x_{\ell}) + \min_{x_{j-1} \in \mathcal{L}} \left\{ \mathrm{OPT}(j-1, x_{j-1}) + \psi_{j-1, j}(x_{j-1}, x_{\ell}) \right\}.
\end{equation*}
\end{proposition}
\begin{proof}
By induction on the elements $j \in \N$.
Induction basis $j = 1$.
Thus, $\mathrm{OPT}(1, \ell) = \varphi_{1}(x_{\ell})$ holds.
Inductive step.
Assume it holds for $j-1$.
Then,
\begin{align*}
	\mathrm{OPT}(j, \ell) & = \min_{x \in \mathcal{L}^{j-1}} \; \sum_{i=1}^{j} \varphi_{i}(x_{i}) + \sum_{i=2}^{j} \psi_{i-1, i}(x_{i-1}, x_{i}) \\
	& = \min_{x_{j-1} \in \mathcal{L}} \; \left\{ \mathrm{OPT}(j-1, x_{j-1}) + \varphi_{j}(x_{\ell}) + \psi_{j-1, j}(x_{j-1}, x_{\ell}) \right\} \\
	& = \varphi_{j}(x_{\ell}) + \min_{x_{j-1} \in \mathcal{L}} \; \left\{ \mathrm{OPT}(j-1, x_{j-1}) + \psi_{j-1, j}(x_{j-1}, x_{\ell}) \right\}.
\end{align*} \qed
\end{proof}

It is straightforward to see that Alg.~\ref{alg:dp} correctly computes all values of $\mathrm{OPT}(j, \ell)$ and hence the minimum value of~\eqref{eq:fdmin}.
Its running time is in $\BigO{n \abs{\mathcal{L}}^{2}}$ and its memory requirement in $\BigO{n \abs{\mathcal{L}}}$.
Recovering a minimiser to~\eqref{eq:fdmin} can be done either by a subsequent backward pass or even faster at the cost of additional memory by storing a minimising label $x_{i}$ in each iteration.

One can generalise the algorithm to energies involving higher-order terms of $q \ge 2$ consecutive unknowns yielding a running time of $\BigO{n \abs{\mathcal{L}}^{q}}$ and a memory requirement of $\BigO{n \abs{\mathcal{L}}^{q-1}}$.
We refer to~\cite{FelZab11} for the details.
The minimisation problems encountered in Sec.~\ref{sec:model} involve terms of order at most three.
It is straightforward to apply the above framework to \eqref{eq:ldj1}, \eqref{eq:ldj2}, \eqref{eq:lpdj1}, and \eqref{eq:lpdj2}.

\begin{algorithm}[t]
\KwIn{Integer $n$, functions $\varphi_{j}$ and $\psi_{j-1, j}$.}
\KwOut{$E(x^{*})$.}
$B[1][\ell] \gets \varphi_{1}(x_{\ell})$, $\forall \ell \in \mathcal{L}$\;
\For{$j \gets 2$ \KwTo $n$}{
	\For{$\ell \gets 1$ \KwTo $\abs{\mathcal{L}}$}{
		$B[j][\ell] \gets \varphi_{j}(x_{\ell}) + \min_{x_{j-1} \in \mathcal{L}} \{ B[j-1][x_{j-1}] + \psi_{j-1, j}(x_{j-1}, x_{\ell}) \}$\;
	}
}
\KwRet{$\min_{\ell} B[n][\ell]$\;}
\caption{Dynamic programming on a sequence.}
\label{alg:dp}
\end{algorithm}

\paragraph{\textbf{Energy Minimisation on Graphs.}}
A more general point of view is to consider \eqref{eq:fdmin} on (undirected) graphs.
Thereby, each element is associated with a vertex of the graph and one seeks a minimiser to
\begin{equation}
	\min_{x \in \mathcal{L}^{n}} \; \sum_{i} \varphi_{i}(x_{i}) + \sum_{i \sim j} \psi_{i, j}(x_{i}, x_{j}).
\label{eq:energy}
\end{equation}
The first term sums over all $n$ vertices whereas the second term sums over all pairs of vertices which are connected via an edge in the graph.
Such energies typically arise from the discretisation of variational problems on regular grids, such as for instance image denoising.
In general, \eqref{eq:energy} is NP-hard.
However, under certain restrictions on $\varphi_{i}$ and $\psi_{i, j}$, there exist polynomial-time algorithms.
Whenever the underlying structure is a tree or a sequence as above, the problem can be solved by means of DP without further restrictions.
See \cite{BlaKohRot11} for a general introduction to the topic.

Nevertheless, many interesting problems do fall in neither category.
One remedy is \emph{block coordinate descent} \cite{CheKol14,MenHeiGei15} in order to approximately minimise \eqref{eq:energy}, which we denote by $F$.
The essential idea is to choose in each iteration $t \in \N$ an index set $\mathcal{I}^{(t)} \subset \{1, \dots, n\}$ which contains much less elements than $n$, and to consider \eqref{eq:energy} only with respect to unknowns $x_{i}$, $i \in \mathcal{I}^{(t)}$.
The values of all other unknowns are taken from the previous iteration.
That is, one finds in each iteration
\begin{equation}
	x^{(t)} \coloneqq \argmin_{x_{i}: i \in \mathcal{I}^{(t)}} \; F^{(t)}(x),
\label{eq:bcd}
\end{equation}
where $F^{(t)}$ denotes the energy in~\eqref{eq:energy} with all unknowns not in $\mathcal{I}^{(t)}$ taking the values from $x^{(t-1)}$.
Typically, $\mathcal{I}^{(t)}$ is chosen such that \eqref{eq:bcd} can be solved efficiently.
It is straightforward to see that block coordinate descent generates a sequence $\{ x^{(t)} \}_{t}$ of solutions such that $F$ never increases.
The energy is thus guaranteed to converge to a local minimum with regard to the chosen $\mathcal{I}^{(t)}$.

We perform block coordinate descent for the solution of \eqref{eq:pdj} and iteratively consider instances of \eqref{eq:fdmin}.
Following the ideas in \cite{CheKol14,MenHeiGei15}, we consecutively minimise in each iteration over all odd columns, all even columns, all odd rows, and finally over all even rows.
During minimisation the displacements of all other rows, respectively columns, are fixed.
See Table~\ref{tab:summary} for the resulting algorithms.
\section{Numerical Results} \label{sec:experiments}

We present qualitative results on the basis of a test image\footnote{Taken from BSDS500 dataset at \url{https://www.eecs.berkeley.edu/Research/Projects/CS/vision/grouping/resources.html}} and create image degraded by line jitter, line pixel jitter, and pixel jitter by sampling displacements from a Gaussian distribution with variance $\sigma^{2} = 1.5$, and rounding them to the nearest integer, see Fig.~\ref{fig:jitter}.
Vector-valued displacements are sampled component-wise.
In addition, we create instances with additive Gaussian white noise $\eta$ of variance $\sigma_{\eta}^{2} = 0.01$.
We then apply to each instance the appropriate algorithm with varying parameters $\alpha$ and $p$, and show the results which maximise jitter removal, see Figs.~\ref{fig:ldj}, \ref{fig:lpdj}, and \ref{fig:pdj}.
The value of $\rho$ is set to the maximum displacement that occurred during creation.
A Matlab/Java implementation is available online.\footnote{\url{https://www.csc.univie.ac.at}}

While regularisation does not seem to be crucial for line jitter removal---apart from a centred image---for line pixel jitter and pixel jitter removal it is.
Moreover, including second-order derivatives in our models does not necessarily improve the result.
In all cases our results indicate effective removal of jitter, even in the presence of moderate noise.
However, subsequent denoising seems obligatory.

\begin{figure}[t]
	\centering
	\includegraphics[width=0.24\textwidth]{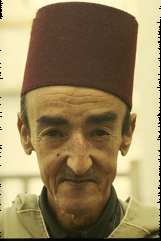}
	\hfill
	\includegraphics[width=0.24\textwidth]{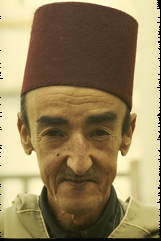}
	\hfill
	\includegraphics[width=0.24\textwidth]{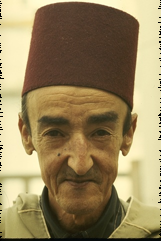}
	\hfill
	\includegraphics[width=0.24\textwidth]{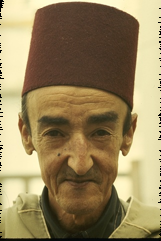} \\ \vspace{0.5em}
		\includegraphics[width=0.24\textwidth]{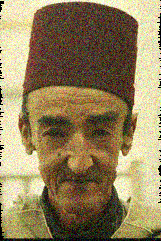}
	\hfill
	\includegraphics[width=0.24\textwidth]{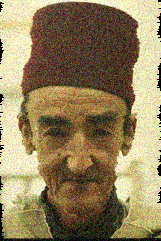}
	\hfill
	\includegraphics[width=0.24\textwidth]{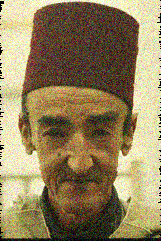}
	\hfill
	\includegraphics[width=0.24\textwidth]{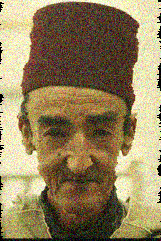}
\caption{Line jitter removal. Top row: minimisers of $\mathcal{E}_{0, 0.5}^{1}$, $\mathcal{E}_{0, 0.5}^{2}$, $\mathcal{E}_{0.01, 0.5}^{1}$, and $\mathcal{E}_{0.01, 0.5}^{2}$, cf. \eqref{eq:ldj1} and \eqref{eq:ldj2}. Bottom row: minimisers of $\mathcal{E}_{0, 3}^{1}$, $\mathcal{E}_{0, 3}^{2}$, $\mathcal{E}_{1000, 3}^{1}$, and $\mathcal{E}_{1000, 3}^{2}$ when noise $\eta$ is present in each of the four $u^{\delta}$.}
\label{fig:ldj}
\end{figure}

\begin{figure}[t]
	\centering
	\includegraphics[width=0.24\textwidth]{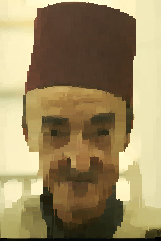}
	\hfill
	\includegraphics[width=0.24\textwidth]{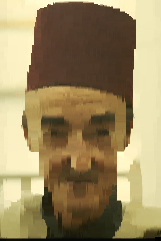}
	\hfill
	\includegraphics[width=0.24\textwidth]{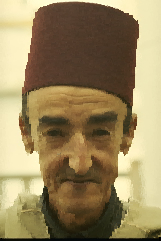}
	\hfill
	\includegraphics[width=0.24\textwidth]{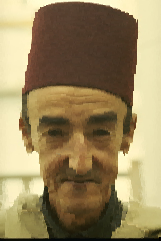} \\ \vspace{0.5em}
	\includegraphics[width=0.24\textwidth]{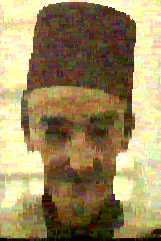}
	\hfill
	\includegraphics[width=0.24\textwidth]{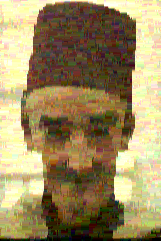}
	\hfill
	\includegraphics[width=0.24\textwidth]{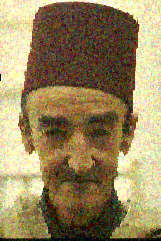}
	\hfill
	\includegraphics[width=0.24\textwidth]{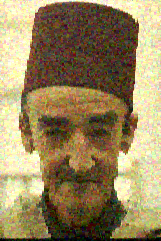}
\caption{Line pixel jitter removal. Top row: minimisers of $\mathcal{F}_{0, 0.5}^{1}$, $\mathcal{F}_{0, 0.5}^{2}$, $\mathcal{F}_{4, 0.5}^{1}$, and $\mathcal{F}_{4, 0.5}^{2}$, cf. \eqref{eq:lpdj1} and \eqref{eq:lpdj2}. Bottom row: minimisers of $\mathcal{F}_{0, 0.5}^{1}$, $\mathcal{F}_{0, 0.5}^{2}$, $\mathcal{F}_{5, 0.5}^{1}$, and $\mathcal{F}_{5, 0.5}^{2}$ when noise $\eta$ is present in each of the four $u^{\delta}$.}
\label{fig:lpdj}
\end{figure}

\begin{figure}[t]
	\centering
	\includegraphics[width=0.24\textwidth]{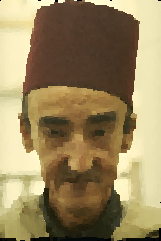}
	\hfill
	\includegraphics[width=0.24\textwidth]{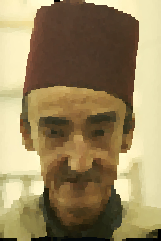}
	\hfill
	\includegraphics[width=0.24\textwidth]{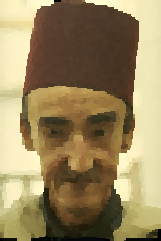}
	\hfill
	\includegraphics[width=0.24\textwidth]{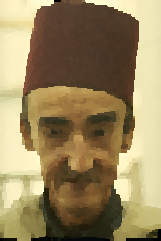} \\ \vspace{0.5em}
	\includegraphics[width=0.24\textwidth]{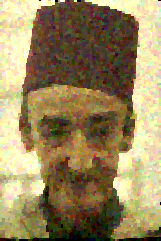}
	\hfill
	\includegraphics[width=0.24\textwidth]{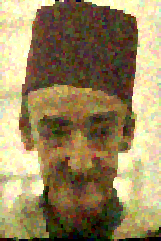}
	\hfill
	\includegraphics[width=0.24\textwidth]{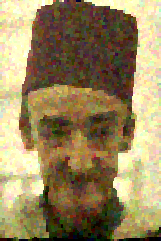}
	\hfill
	\includegraphics[width=0.24\textwidth]{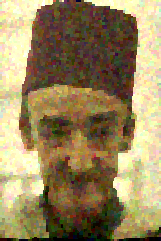} 
\caption{Line pixel jitter removal. Top row: sequence $u^{(1)}, \dots, u^{(4)}$, of minimisers of $\mathcal{G}_{4, 0.5}$, cf. \eqref{eq:pdj}. Bottom row: sequence $u^{(1)}, \dots, u^{(4)}$, of minimisers of $\mathcal{G}_{4, 0.5}$ when noise $\eta$ is present in $u^{\delta}$. Results for $\alpha = 0$ are excluded due to poor quality.}
\label{fig:pdj}
\end{figure}
\section{Related work}

Kokaram~et~al. \cite{KokRay92,KokRayVanBie97,Kok98} were among the first to consider line dejittering.
They employed a block-based autoregressive model for grey-valued images and developed an iterative multi-resolution scheme to estimate line displacements.
Subsequent drift compensation removes low frequency oscillations.
Their methods seek displacements which reduce the vertical gradients of the sought image.

For line jitter removal, a naive approach consists of fixing the first line of the image and successive minimisation of a mismatch function between consecutive lines.
This greedy algorithm tends to introduce vertical lines in the reconstructed image and fails in the presence of dominant non-vertical edges \cite{Kok98,Lab03}.
As a remedy, Laborelli \cite{Lab03} proposed to apply DP and to recover a horizontal displacement for each line by minimising the sum of the pixel-wise differences between two or three consecutive lines.

Shen~\cite{She04} proposed a variational model for line dejittering in a Bayesian framework and investigated its properties for images in the space of bounded variation.
In order to minimise the non-linear and non-convex objective, an iterative algorithm that alternatively estimates the original image and the displacements is devised.

Kang and Shen~\cite{KanShe06} proposed a two-step iterative method termed ``bake and shake''.
In a first step, a Perona-Malik-type diffusion process is applied in order to suppress high-frequency irregularities in the image and to smooth distorted object boundaries.
In a second step, the displacement of each line is independently estimated by solving a non-linear least-squares problem.
In another article~\cite{KanShe07}, they investigated properties of slicing moments of images with bounded variation and proposed a variational model based on moment regularisation.

In \cite{Nik09a,Nik09}, Nikolova considered greedy algorithms for finite-dimensional line dejittering with bounded displacements.
These algorithms consider vertical differences between consecutive lines up to third-order and are applicable to grey-value as well as colour images.
In each step, a non-smooth and possibly non-convex function is minimised by enumeration, leading to an $\BigO{mn\rho}$ algorithm.
For experimental evaluation various error measures are discussed.

Lenzen and Scherzer~\cite{LenSch09,LenSch11} considered partial differential equations for displacement error correction in multi-channel data.
Their framework is applicable to image interpolation, dejittering, and deinterlacing.
For line pixel dejittering they derived gradient flow equations for a non-convex variational formulation involving the total variation of the reconstructed image.

Dong~et~al.~\cite{DonPatSchOek15} treated the finite-dimensional models in~\cite{Nik09a,Nik09} in an infinite-dimensional setting.
They derived the corresponding energy flows and systematically investigated their applicability for various jitter problems.

\section{Conclusion}

In this article we presented efficient algorithms for line jitter, line pixel jitter, and pixel jitter removal in a finite-dimensional setting.
By assuming (approximate) invertibility of the image acquisition equation we were able to cast the minimisation problems into a well-known DP framework.
Our experimental results indicate effective removal of jitter, even in the presence of moderate noise.

\def\cprime{$'$}
  \providecommand{\noopsort}[1]{}\def\ocirc#1{\ifmmode\setbox0=\hbox{$#1$}\dimen0=\ht0
  \advance\dimen0 by1pt\rlap{\hbox to\wd0{\hss\raise\dimen0
  \hbox{\hskip.2em$\scriptscriptstyle\circ$}\hss}}#1\else {\accent"17 #1}\fi}

\end{document}